\documentclass[runningheads]{llncs}

\usepackage{amsmath, amssymb, enumerate ,mathtools,mathrsfs,xfrac,enumitem}
\usepackage{amsfonts}
\usepackage{enumerate}

\usepackage{fontawesome}

\newenvironment{proofof}[1]{\par\noindent{\bfseries\upshape Proof of #1\ }}{\qed}

\newcommand{\reals}{\mathbb{R}}

\newcommand{\posreals}{\reals_{> 0}}
\newcommand{\nnegreals}{\reals_{\geq 0}}
\newcommand{\nats}{\mathbb{N}}
\newcommand{\natz}{\mathbb{N}_{0}}

\newcommand{\indica}[1]{\mathbb{I}_{#1}}

\newcommand{\prev}{\mathrm{E}}

\newcommand{\upprev}{\overline{\mathrm{E}}}

\newcommand{\lowprev}{\underline{\mathrm{E}}}

\newcommand{\lupprev}[1]{\overline{\mathrm{Q}}_{#1}}

\newcommand{\settrans}{\mathscr{T}}

\newcommand{\statespace}{\mathscr{X}}

\newcommand{\samplespace}{\Omega}

\newcommand{\setofgambles}{\mathscr{L}}

\newcommand{\setoffingambles}{\mathscr{L}_{\mathrm{\/ fin}}}

\newcommand{\uptrans}{\overline{T}}
\newcommand{\lowtrans}{\underline{T}}
\newcommand{\probability}[2]{\mathrm{P}(#1\vert#2)}
\newcommand{\setofprocesses}[1]{\mathscr{P}_{#1}}

\begin{document}

\title{A Recursive Algorithm for Computing Inferences in Imprecise Markov Chains}
\titlerunning{Computing Inferences in Imprecise Markov Chains}

\author{\Large Natan T'Joens
%\orcidID{} 
\and Thomas Krak
%\orcidID{} 
\and Jasper De Bock
%\orcidID{} 
\and Gert de Cooman
%\orcidID{}
}

\authorrunning{N. T'Joens et al.}

\institute{ELIS -- FLip, Ghent University, Belgium
\email{\{natan.tjoens,thomas.krak,jasper.debock,gert.decooman\}@ugent.be}}

\maketitle

\begin{abstract}
We present an algorithm that can efficiently compute a broad class of inferences for discrete-time imprecise Markov chains, a generalised type of Markov chains that allows one to take into account partially specified probabilities and other types of model uncertainty. The class of inferences that we consider contains, as special cases, tight lower and upper bounds on expected hitting times, on hitting probabilities and on expectations of functions that are a sum or product of simpler ones. Our algorithm exploits the specific structure that is inherent in all these inferences: they admit a general recursive decomposition.
This allows us to achieve a computational complexity that scales linearly in the number of time points on which the inference depends, instead of the exponential scaling that is typical for a naive approach.
%Finally, we study the limit behaviour of these methods, and give some indications about how our results translate to inferences that depend on infinite time intervals.%.
%Due to recently developed continuity results for upper and lower expectations, we will then be able to extend the scope of our methods towards inferences that span an ininite time interval.

 % Changed "expected time averages" to "time averages" to keep the keywords on two lines
\end{abstract}

\section{Introduction}

 Markov chains are popular probabilistic models for describing the behaviour of dynamical systems under uncertainty. The crucial simplifying assumption in these models is that the probabilities describing the system's future behaviour are conditionally independent of its past behaviour, given that we know the current state of the system; this is the canonical \emph{Markov property}. 
 
 It is this Markov assumption that makes the parametrisation of a Markov chain relatively straightforward---indeed, as we will discuss in Section~\ref{sec3:Markov}, the uncertain dynamic behaviour is then completely characterised by a transition matrix $T$, whose elements $T(x_n,x_{n+1})= \probability{X_{n+1}=x_{n+1}}{X_n=x_n}$ describe the probabilities that the system will transition from any state $x_n$ at time $n$, to any state $x_{n+1}$ at time $n+1$. Note that $T$ itself is independent of the time $n$; this is the additional assumption of \emph{time homogeneity} that is often imposed implicitly in this context. An important advantage of these assumptions is that the resulting matrix $T$ can be used to solve various important inference problems, using one of the many available efficient algorithms.

In many cases however, the numerical value of the transition matrix $T$ may not be known exactly; that is, there may be additional (higher-order) uncertainty about the model itself. Moreover, it can be argued that simplifying assumptions like the Markov property and time homogeneity are often unrealistic in practice. It is of interest, then, to compute inferences in a manner that is \emph{robust}; both to violations of such simplifying assumptions, and to variations in the numerical values of the transition probabilities.

The theory of \emph{imprecise probabilities} allows us to describe such additional uncertainties by using, essentially, \emph{sets} of traditional (``precise'') models. In particular, such a set is comprised of all the models that we deem ``plausible''; for instance, we may include all Markov chains whose characterising transition matrix $T$ is included in some given \emph{set} $\settrans$ of transition matrices. In this way we can also include non-homogeneous Markov chains, by simply requiring that their (now time-dependent) transition matrices remain in $\settrans$. Moreover, we can even include non-Markovian models in such a set. 
This leads to the notion of an \emph{imprecise Markov chain}. 
The robust inferences that we are after, are then the tightest possible lower and upper bounds on the inferences computed for each of the included precise models.

In this work, we present an efficient algorithm for solving a large class of these inferences within imprecise Markov chains. 
Broadly speaking, this class consists of inferences that depend on the uncertain state of the system at a finite number of time instances, and which can be decomposed in a particular recursive form. 
As we will discuss, it contains as special cases the \emph{(joint) probabilities} of sequences of states; the \emph{hitting probabilities} and \emph{expected hitting times} of subsets of the possible states; and \emph{time averages} of functions of the state of the system.

Interestingly, existing algorithms for some of these inferences turn out to correspond to special cases of our algorithm, giving our algorithm a unifying character. Time averages, for example, were already considered in~\cite{8535240}, and some of the results in~\cite{Krak2019HittingTimesISIPTA}---a theoretical study of lower and upper expected hitting times and probabilities---can be interpreted as a special cases of the algorithm presented here. Readers that are familiar with recursive algorithms for credal networks under epistemic irrelevance~\cite{DeBock2017IJAR,debock2015credal,deCooman:2010gd} might also recognise some of what we do; in fact, many of the ideas behind our algorithm have previously been discussed in this more general context~\cite[Chapter 7]{debock2015credal}. 

% In order to adhere to the page limit, all proofs have been relegated to the appendix of an online extended version~\cite{tjoens2019extended}. 

In order to facilitate the reading, proofs and intermediate results are relegated to the appendix.

\section{Preliminaries}\label{sec3:Markov}

We denote the natural numbers, without $0$, by $\nats$, and let $\natz \coloneqq \nats \cup \{0\}$.  
The set of positive real numbers is denoted by $\posreals$ and the set of non-negative real numbers by $\nnegreals$. 
%The set of extended real numbers is denoted by $\extreals \coloneqq \reals \cup \{- \infty, + \infty\}$. 
Throughout, we let $\indica{A}$ denote the indicator of any subset $A\subseteq \mathscr{Y}$ of a set $\mathscr{Y}$; so, for any $y\in\mathscr{Y}$, $\indica{A}(y)\coloneqq 1$ if $y\in A$ and $\indica{A}(y)\coloneqq 0$ otherwise.

Before we can introduce the notion of an imprecise Markov chain, we first need to discuss general (non-Markovian) stochastic processes. These are arguably most commonly formalised using a measure-theoretic approach; however, the majority of our results do not require this level of generality, and so we will keep the ensuing introduction largely intuitive and informal. %however, for technical reasons our results are instead based on a game-theoretic probability framework for describing uncertain processes. It is worth pointing out, however, that these two formalisms typically lead to the same inferences when {\bf ** blabla **}.
%To this end, we will keep the following introduction fairly general, and will highlight at any point where care should be taken about which framework to use.

Let us start by considering the \emph{realisations} of a stochastic process. 
At each point in time $n\in\nats$, such a process is in a certain \emph{state} $x_n$, which is an element of a \emph{finite} non-empty state space $\statespace$. 
A realisation of the process is called a \emph{path}, and is an infinite sequence $\omega = x_1 x_2 x_3 \cdots$ where, at each discrete time point $n\in\nats$, $\omega_n\coloneqq x_n\in\statespace$ is the state obtained by the process at time $n$, on the path $\omega$. 
So, we can interpret any path as a map $\omega:\nats\to\statespace$, allowing us to collect all paths in the set $\samplespace\coloneqq \statespace{}^{\nats}$. Moreover, for any $\omega\in\samplespace$ and any $m,n\in\nats$ with $m\leq n$, we use the notation $\omega_{m:n}$ to denote the finite sequence of states $\omega_m\cdots\omega_n\in\statespace^{n-m+1}$.

A stochastic process is now an infinite sequence $X_1X_2X_3\cdots$ of uncertain states where, for all $n\in\nats$, the uncertain state at time $n$ is a function of the form  $X_n:\samplespace\rightarrow\statespace:\omega\mapsto\omega_n$. 
Similarly, we can consider finite sequences of such states where, for all $m,n\in\nats$ with $m\leq n$, $X_{m:n}:\samplespace\rightarrow \statespace^{n-m+1}:\omega\mapsto \omega_{m:n}$.  
These states are uncertain in the sense that we do not know which realisation $\omega\in\Omega$ will obtain in reality; rather, we assume that we have assessments of the probabilities $\probability{X_{n+1} = x_{n+1}}{X_{1:n} = x_{1:n}}$, for any $n \in \nats$ and any $x_{1:n} \in \statespace^{n}$. 
Probabilities of this form tell us something about which state the process might be in at time $n+1$, given that we know that at time points $1$ through $n$, it followed the sequence $x_{1:n}$. 
Moreover, we can consider probabilities of the form $\mathrm{P}(X_1=x_1)$ for any $x_1\in\statespace$; this tells us something about the state that the process might start in. 
It is well known that, taken together, these probabilities suffice to construct a global probability model for the \emph{entire} process $X_1X_2X_3\cdots$, despite each assessment only being about a finite subsequence of the states; see e.g. the discussion surrounding~\cite[Theorem 5.16]{kallenberg2006foundations} for further details on formalising this in a proper measure-theoretic setting. We simply use $\mathrm{P}$ to denote this global model.

Once we have such a global model $\mathrm{P}$, we can talk about \emph{inferences} in which we are interested. In general, these are typically encoded by functions $f:\Omega\to\reals$ of the unknown realisation $\omega\in\Omega$, and we collect all functions of this form in the set $\setofgambles(\samplespace)$. 
To compute such an inference consists in evaluating the (conditional) expected value $\prev{}_\mathrm{P}(f \vert \,C)$ of $f$ with respect to the model $\mathrm{P}$, where $C\subseteq\Omega$ is an event of the form $X_{m:n} = x_{m:n}$ with $m,n \in \nats{}$ such that $m \leq n$. 
In particular, if $\mathrm{P}$ is a global model in the measure-theoretic sense, then under some regularity conditions like the measurability of $f$, we would be interested in computing the quantity $\prev{}_\mathrm{P}(f\vert\,C) \coloneqq \int_{\samplespace}f(\omega)\,\mathrm{d}\mathrm{P}(\omega\vert C)$. 
For notational convenience, we will also use $X_{1:0} \coloneqq \samplespace{}$ as a trivial conditioning event, allowing us to regard unconditional expectations as a special case of conditional ones. 
%When $\mathrm{P}$ is instead a global model in the game-theoretic probability sense, the definition of this expectation is slightly more involved; see e.g. \cite{Shafer:2005wx,Tjoens2019NaturalExtensionISIPTA} for details. 

A special type of inferences that will play an important role in the remainder of this work are those for which the function $f$ only depends on a \emph{finite} subsequence of the path $\omega$, thereby vastly simplifying the definition of its expectation. 
In particular, if an inference only depends on the states at time points $m$ through $n$, say, then it can always be represented by a function $f:\statespace^{n-m+1}\to\reals$ evaluated in the uncertain states $X_{m:n}$; specifically, the inference is represented by the composition $f\circ X_{m:n}$, which we will denote by $f(X_{m:n})$. 
In the sequel, we will call a composite function of this form \emph{finitary}. 
Moreover, for any $n\in\nats$, we denote by $\setofgambles(\statespace^n)$ the set of all functions of the form $f:\statespace^n\to\reals$, and we write $\setoffingambles{}(\samplespace{})\subset\setofgambles(\samplespace)$ for the set of all finitary functions. 
For a finitary function $f(X_{1:n})$, the computation of its expected value reduces to evaluating the finite sum%~{\bf REFs}
\begin{equation*}\label{eq:expectation_finitary_function}
\prev{}_\mathrm{P}(f(X_{1:n})\vert\,C) = \sum_{x_{1:n}\in\statespace^{n}}f(x_{1:n})\mathrm{P}(X_{1:n}=x_{1:n}\vert\,C)\,.
\end{equation*}

%Let us now consider how to compute the sum in~\eqref{eq:expectation_finitary_function}. Clearly, for any $x_{0:n}\in\statespace^{n+1}$, we need to compute the probability $\mathrm{P}(X_{0:n}=x_{0:n})$. It follows from the basic laws of probability that, to compute this quantity, we can make use of the probability assessments that we assumed to know before:
%\begin{equation*}
%\mathrm{P}(X_{0:n}=x_{0:n}) = \mathrm{P}(X_0=x_0)\prod_{k=1}^n \mathrm{P}(X_{0:k}=x_{0:k}\vert X_{0:k-1}=x_{0:k-1})\,. 
%\end{equation*}

Let us now move on from the discussion about general uncertain processes, to the special case of Markov chains. An uncertain process $\mathrm{P}$ is said to satisfy the \emph{Markov property} if, for all $n\in\nats$ and all $x_{1:n+1}\in\statespace^{n+1}$, the aforementioned probability assessments simplify in the sense that
\begin{equation*}
\probability{X_{n+1}=x_{n+1}}{X_{1:n}=x_{1:n}} = \probability{X_{n+1}=x_{n+1}}{X_n=x_n}\,.
\end{equation*}
A process that satisfies this Markov property is called a \emph{Markov chain}. Thus, for a Markov chain, the probability that it will visit state $x_{n+1}$ at time $n+1$ is independent of the states $X_{1:n-1}$, given that we know the state $X_n$ at time $n$. 
If the process is moreover \emph{homogeneous}, meaning that $\probability{X_{n+1}=y}{X_n=x}=\probability{X_{2}=y}{X_1=x}$ for all $x,y\in\statespace$ and all $n\in\nats$, then the parameterisation of the process becomes exceedingly simple. 
Indeed, up to the initial distribution $\mathrm{P}(X_1)$---a probability mass function on $\statespace$---the process' behaviour is then fully characterised by a single $\lvert\statespace\rvert \times \lvert\statespace\rvert$ matrix $T$ that is called the transition matrix. It is row-stochastic (meaning that, for all $x\in\statespace$, the $x$-th row $T(x,\cdot)$ of $T$ is a probability mass function on $\statespace$) and its entries satisfy
$T(x,y) = \probability{X_{n+1}=y}{X_n=x}$ for all $x,y\in\statespace$ and $n\in\nats$.
The usefulness of this representation comes from the fact that we can interpret $T$ as a linear operator on the vector space $\setofgambles{(\statespace)}\simeq \reals^{\lvert\statespace\rvert}$, due to the assumption that $\statespace$ is finite. 
For $f\in\setofgambles{(\statespace)}$, this allows us to write the conditional expectation of $f(X_{n+1})$ given $X_n$ as a matrix-vector product: for any $x\in\statespace$, $\prev{}_\mathrm{P}(f(X_{n+1})\vert X_n=x)$ equals
\begin{equation*}
\sum_{y\in\statespace} f(y)\probability{X_{n+1}=y}{X_n=x} = \sum_{y\in\statespace} f(y)T(x,y) = \bigl[Tf\bigr](x).
\end{equation*}
%***Eventueel dit ook nog weg?***
% Moreover, using the Markov property and time-homogeneity, it is easily verified that, for any $n\in\nats$, we can write the conditional expectation of $f(X_n)$ given $X_1$ using the $(n-1)$-th matrix power of $T$, \emph{viz}. $\prev{}_\mathrm{P}(f(X_n)\,\vert\,X_1=x) = \bigl[T^{n-1}f\bigr](x)$. {\bf ** I don't think we use this property anywhere, So we could remove this last paragraph without loss apart from providing context **}

\section{Imprecise Markov chains}\label{section: imprecise markov chains}

Let us now move on to the discussion about \emph{imprecise} Markov chains. Here, we additionally include uncertainty about the model specification, such as uncertainty about the numerical values of the probabilities $\probability{X_{n+1}}{X_{1:n}}$, and about the validity of structural assessments like the Markov property. 

We will start this discussion by regarding the parameterisation of such an imprecise Markov chain. 
We first consider the (imprecise) \emph{initial model} $\mathcal{M}$; this is simply a non-empty set of probability mass functions on $\statespace$ that we will interpret as containing those probabilities that we deem to plausibly describe the process starting in a certain state.
Next, instead of being described by a single transition matrix $T$, an imprecise Markov chain's dynamic behaviour is characterised by an entire \emph{set} $\settrans$ of transition matrices. 
So, each element $T\in\settrans$ is an $\lvert\statespace\rvert\times\lvert\statespace\rvert$ matrix that is row-stochastic. 
In the sequel, we will take $\settrans$ to be fixed, and assume that it is non-empty and that it has separately specified rows. 
This last property is instrumental in ensuring that computations can be performed efficiently, and is therefore often adopted in the literature; see e.g.~\cite{itip:stochasticprocesses} for further discussion. For our present purposes, it suffices to know that it means that $\settrans$ can be completely characterised by providing, for any $x\in\statespace$, a non-empty set $\settrans_x$ of probability mass functions on $\statespace$. In particular, it means that $\settrans$ is the set of all row-stochastic $\lvert\statespace\rvert\times\lvert\statespace\rvert$ matrices $T$ such that, for all $x\in\statespace$, the $x$-row $T(x,\cdot)$ is an element of $\settrans_x$.
%for any $x\in\statespace$ and any $T,S\in\settrans$, there exists a $Q\in\settrans$ such that the $x$-row of $Q$ equals that of $T$, i.e. $Q(x,\cdot)=T(x,\cdot)$, and which otherwise agrees with $S$, i.e. $Q(y,\cdot)=S(y,\cdot)$ for all $y\in\statespace$ such that $y\neq x$.
%This last property means that, for any $T,S\in\settrans$ and any diagonal matrix $\Lambda$ such that $\Lambda(x,x)\in[0,1]$ for all $x\in\statespace$, the matrix $\Lambda T + (I-\Lambda) S$ is also included in $\settrans$; here $I$ denotes the identity matrix. 
% Note that, because of the assumed closedness and convexity of $\settrans$, the sets $\settrans_x$ are closed and convex as well. 
%, and can be defined independently from any other set $\settrans_y$ with $y\neq x$; this last property follows directly from $\settrans$ having separately specified rows.

Given the sets $\mathcal{M}$ and $\settrans$, the corresponding imprecise Markov chain is defined as the largest \emph{set} $\setofprocesses{\mathcal{M},\settrans{}}$ of stochastic processes that are in a specific sense \emph{compatible} with both $\mathcal{M}$ and $\settrans$. 
%In particular, this is the ``sensitivity analysis'' interpretation of an imprecise probability model. 
In particular, a model $\mathrm{P}$ is said to be compatible with $\mathcal{M}$ if $\mathrm{P}(X_1)\in\mathcal{M}$, and it is said to be compatible with $\settrans$ if, for all $n\in\nats$ and all $x_{1:n}\in\statespace^{n}$, there is some $T\in\settrans$ such that
\begin{equation*}
\probability{X_{n+1}=x_{n+1}}{X_{1:n}=x_{1:n}} = T(x_n,x_{n+1}) \text{ for all } x_{n+1} \in \statespace{}.
\end{equation*} 
Notably, therefore, $\setofprocesses{\mathcal{M},\settrans{}}$ contains all the (precise) homogeneous Markov chains whose characterising transition matrix $T$ is included in $\settrans$, and whose initial distribution $\mathrm{P}(X_1)$ is included in $\mathcal{M}$.
However, in general,  $\setofprocesses{\mathcal{M},\settrans{}}$ clearly also contains models that do not satisfy the Markov property, as well as Markov chains that are not homogeneous.\footnote{Within the field of imprecise probability theory, this model is called an \emph{imprecise Markov chain under epistemic irrelevance}~\cite{deCooman:2009jz,itip:stochasticprocesses,8535240}.}

For such an imprecise Markov chain, we are interested in computing inferences that are in a specific sense robust with respect to variations in the set $\setofprocesses{\mathcal{M},\settrans{}}$. 
Specifically, for any function of interest $f:\samplespace\to\reals$, we consider its (conditional) \emph{lower} and \emph{upper} expectations, which are respectively defined by
\begin{equation*}
\vspace*{-1pt}
\lowprev{}_{\mathcal{M},\settrans{}}(f \vert \,C) \coloneqq \inf_{\mathrm{P} \in \setofprocesses{\mathcal{M},\settrans{}}} \prev{}_\mathrm{P}(f \vert \,C)
\quad\text{and}\quad
\upprev{}_{\mathcal{M},\settrans{}}(f \vert \,C) \coloneqq \sup_{\mathrm{P} \in \setofprocesses{\mathcal{M},\settrans{}}} \prev{}_\mathrm{P}(f \vert \,C)\,.
\vspace*{-1pt}
\end{equation*}
In words, we are interested in computing the tightest possible bounds on the inferences computed for each $\mathrm{P} \in \setofprocesses{\mathcal{M},\settrans{}}$. 
These lower and upper expectations are related through conjugacy, meaning that $\lowprev{}_{\mathcal{M},\settrans{}}(f \vert \,C) = -\upprev{}_{\mathcal{M},\settrans{}}(-f \vert \,C)$, so it suffices to consider only the upper expectations in the remaining discussion; any results for lower expectations follow analogously through this relation.

From a computational point of view, it is also useful to consider the dual representation of the set $\settrans$, given by the \emph{upper transition operator} $\uptrans$ with respect to this set~\cite{deCooman:2009jz,itip:stochasticprocesses}. 
This is a (non-linear) operator that maps $\setofgambles(\statespace)$ into $\setofgambles(\statespace)$; it is defined for any $f\in \setofgambles(\statespace)$ and any $x\in\statespace$ as
\begin{equation*}
\bigl[\uptrans f\bigr](x) \coloneqq \sup_{T(x,\cdot)\in\settrans_x} \sum_{y\in\statespace}T(x,y)f(y)%\quad\text{where}\quad \settrans_x\coloneqq\bigl\{ T(x,\cdot)\,\big\vert\,T\in\settrans \bigr\}\,
.
\end{equation*}
So, in order to evaluate $\bigl[\uptrans f\bigr](x)$, one must solve an optimisation problem over the set $\settrans_x$ containing the $x$-rows of the elements of $\settrans$.
In many practical cases, the set $\settrans_x$ is closed and convex and therefore, evaluating $\bigl[\uptrans f\bigr](x)$ is relatively straightforward: for instance, if $\settrans_x$ is described by a finite number of (in)equality constraints, then this problem reduces to a simple linear programming task, which can be solved by standard techniques. We will also make use of the conjugate \emph{lower transition operator} $\lowtrans$, defined by $[\lowtrans\,f](x)\coloneqq -[\uptrans (-f)](x)$ for all $x\in\statespace$ and all $f\in\setofgambles(\statespace)$. Results about upper transition operators translate to results about lower transition operators through this relation; we will focus on the former in the following discussion.

Now, the operator $\uptrans$ can be used for computing upper expectations in much the same way as transition matrices are used for computing expectations with respect to precise Markov chains: for any $n \in \nats{}$, any finitary function $f(X_{n+1})$ and any $x_{1:n}\in\statespace^{n}$ it holds that
\begin{equation}\label{Eq: Markov property}
\upprev{}_{\mathcal{M},\settrans{}}(f(X_{n+1}) \vert X_{1:n}=x_{1:n}) = \bigl[\uptrans f\bigr](x_n);
\end{equation}
see appendix \ref{appendix}.
Observe that the right-hand side in this expression does not depend on the history $x_{1:n-1}$; this can be interpreted as saying that the model satisfies an \emph{imprecise Markov property}, which explains why we call our model an ``imprecise Markov chain''.
Moreover, a slightly more general property holds that will be useful later on:
% Here, as well as further on, we use, for any $k,\ell \in \natz$ such that $k \leq \ell$, the notation $\setofgambles{}(\statespace{}^{\ell - k +1})$ to denote all real-valued functions on $\statespace{}^{\ell - k +1}$.

\begin{proposition}\label{prop:time shift}
Consider the imprecise Markov chain $\setofprocesses{\mathcal{M},\settrans}$. 
For any $m,n \in \nats$ such that $m \leq n$, any function $f \in \setofgambles{}(\statespace{}^{n - m +1})$ and any $x_{1:m-1}\in\statespace^{m-1}$ and $y\in\statespace$, we have that 
\begin{align*}
\upprev{}_{\mathcal{M},\settrans{}}\big(f(X_{m:n}) \big\vert X_{1:m-1} = x_{1:m-1}, X_m=y \big) = \upprev{}_{\mathcal{M},\settrans{}}\big(f(X_{1:n-m+1}) \big\vert X_{1} = y \big).
\end{align*}
\end{proposition}

Finally, we remark that, for any $m,n \in \nats$ such that $m \leq n$, a conditional upper expectation $\upprev{}_{\mathcal{M},\settrans{}}\bigl(f \big\vert X_{m:n}\bigr)$ is itself a (finitary) function depending on the states $X_{m:n}$. Using this observation, we can now introduce the \emph{law of iterated upper expectations}, which will form the basis of the algorithms developed in the following sections:

\begin{theorem}\label{Theorem: law of iterated}
Consider the imprecise Markov chain $\setofprocesses{\mathcal{M},\settrans}$. 
For all $m \in \natz$, all $k \in \nats$ and all $f \in \setoffingambles{}(\samplespace{})$, we have that
\begin{align*}
\upprev{}_{\mathcal{M},\settrans{}}\bigl(f \big\vert X_{1:m} \bigr) = \upprev{}_{\mathcal{M},\settrans{}}\Big(\upprev{}_{\mathcal{M},\settrans{}}\bigl(f \big\vert X_{1:m+k} \bigr) \Big\vert X_{1:m} \Big).
\end{align*}
\end{theorem}

\section{A recursive inference algorithm}\label{section: computing inferences}

% We now procede to present methods to efficiently compute lower and upper expectations of a particular class of functions $f(X_{1:n})$.
%We will start by illustrating the ideas behind these methods using one of the most significant inferences that is encompassed by this class: hitting times.

%\subsection*{Computing Hitting Times}

In principle, for any function $f\in\setofgambles{}(\statespace{}^n)$ with $n \in \nats{}$, the upper expectations of $f(X_{1:n})$ can be obtained by maximising $\prev{}_\mathrm{P}(f(X_{1:n}))$ over the set $\setofprocesses{\mathcal{M},\settrans{}}$ of all precise models $\mathrm{P}$ that are compatible with $\mathcal{M}$ and $\settrans{}$.
Since this will almost always be infeasible if $n$ is large, we usually apply the law of iterated upper expectations in combination with the Markov property in order to divide the optimisation problem into multiple smaller ones.
Indeed, because of Theorem \ref{Theorem: law of iterated}, we have that
\begin{align*}
\upprev{}_{\mathcal{M},\settrans{}}(f(X_{1:n})) = \upprev{}_{\mathcal{M},\settrans{}}\Big(\upprev{}_{\mathcal{M},\settrans{}}(f(X_{1:n})\vert X_{1:n-1}) \Big).
\end{align*}
Using Equation~\eqref{Eq: Markov property}, one can easily show that $\upprev{}_{\mathcal{M},\settrans{}}(f(X_{1:n})\vert X_{1:n-1})$ can be computed by evaluating $[\uptrans{} f(x_{1:n-1} \cdot)](x_{n-1})$ for all $x_{1:n-1} \in \statespace{}^{n-1}$.
Here, $f(x_{1:n-1} \cdot)$ is the function in $\setofgambles{}(\statespace{})$ that takes the value $f(x_{1:n})$ on $x_n \in \statespace{}$.
This accounts for \smash{${\vert \statespace{} \vert}^{n-1}$} optimisation problems to be solved.
With the acquired function $f'(X_{1:n-1}) \coloneqq \upprev{}_{\mathcal{M},\settrans{}}(f(X_{1:n})\vert X_{1:n-1})$, we can then compute the upper expectation $\smash \upprev{}_{\mathcal{M},\settrans{}}(f'(X_{1:n-1})\vert X_{1:n-2})$ in a similar way, by solving \smash{${\vert \statespace{} \vert}^{n-2}$} optimisation problems.
Continuing in this way, we end up with a function that only depends on $X_1$ and for which the expectation needs to be maximised over the initial models in $\mathcal{M}$.
Hence, in total, $\sum_{i = 0}^{n-1} {\vert \statespace{} \vert}^{i}$ optimisation problems need to be solved in order to obtain $\upprev{}_{\mathcal{M},\settrans{}}(f(X_{1:n}))$.
Although these optimisation problems are relatively simple and therefore feasible to solve individually, the total number of required iterations is still exponential in $n$, therefore making the computation of $\upprev{}_{\mathcal{M},\settrans{}}(f(X_{1:n}))$ intractable when $n$ is large.

In many cases, however, $f(X_{1:n})$ can be recursively decomposed in a specific way allowing for a much more efficient computational scheme to be employed; see Theorem~\ref{theorem: algorithm} further on. 
Before we present this scheme in full generality, let us first provide some intuition about its basic working principle.

% Recall that, in the discussion above, we approached the problem by first considering the (inner) upper expectations $\upprev{}_{\mathcal{M},\settrans{}}(f(X_{1:n})\vert X_{1:n-1})$ conditional on events where all but the last state is fixed, then gradually working our way backwards by computing upper expectations with fewer fixed states in the conditioning event.
% On the other hand, one could just as well start by first considering the upper expectations $\upprev{}_{\mathcal{M},\settrans{}}(f(X_{1:n})\vert X_1)$ on the initial situation $X_1$.
% The values of $\upprev{}_{\mathcal{M},\settrans{}}(f(X_{1:n})\vert X_1)$ could then easily be used to compute $\upprev{}_{\mathcal{M},\settrans{}}(f(X_{1:n}))$ by solving a single maximisation problem (with maximally $\ell$ extreme points).
% In order to obtain $\upprev{}_{\mathcal{M},\settrans{}}(f(X_{1:n})\vert X_1)$ we will then typically first calculate the upper expectations $\upprev{}_{\mathcal{M},\settrans{}}(f(X_{1:n})\vert X_{1:2})$, which on its turn necessitates the calculation of $\upprev{}_{\mathcal{M},\settrans{}}(f(X_{1:n})\vert X_{1:3})$ and so on.
% Hence, this approach results in the same computational method as described earlier and will therefore not give us any computational benefit.

So assume we are interested in $\upprev{}_{\mathcal{M},\settrans{}}(f(X_{1:n}))$, which, according to Theorem~\ref{Theorem: law of iterated}, can be obtained by maximising $\prev{}_{\mathrm{P}}(\upprev{}_{\mathcal{M},\settrans{}}(f(X_{1:n})\vert X_1))$ over $\mathrm{P}(X_1)\in\mathcal{M}$.
The problem then reduces to the question of how to compute $\upprev{}_{\mathcal{M},\settrans{}}(f(X_{1:n})\vert X_1)$ efficiently.
Suppose now that $f(X_{1:n})$ takes the following form:
\begin{equation}\label{Eq: structure function}
f(X_{1:n})=g(X_1)+h(X_1)\tau(X_{2:n}), 
\end{equation}
for some $g, h \in \setofgambles{}(\statespace{})$ and some $\tau \in \setofgambles{}(\statespace{}^{n-1})$.
Then, because $\upprev{}_{\mathcal{M},\settrans{}}$ is a supremum over linear expectations, we find that
\begin{align*}
\upprev{}_{\mathcal{M},\settrans{}}(f(X_{1:n})\vert X_1) = g(X_1)+h(X_1) \upprev{}_{\mathcal{M},\settrans{}}(\tau(X_{2:n})\vert X_1),
\end{align*}
where, for the sake of simplicity, we assumed that $h$ does not take negative values.
Then, by appropriately combining Proposition \ref{prop:time shift} with Theorem \ref{Theorem: law of iterated}, one can express $\upprev{}_{\mathcal{M},\settrans{}}(\tau(X_{2:n})\vert X_1)$ in terms of $\overline{\Upsilon}\colon\statespace\to\reals$, defined by  
\begin{equation*}
\overline{\Upsilon}(x)\coloneqq\upprev{}_{\mathcal{M},\settrans{}}(\tau(X_{1:n-1})\vert X_1=x)\text{ for all }x \in \statespace{}. 
\end{equation*}
In particular, we find that
\begin{align*}
\upprev{}_{\mathcal{M},\settrans{}}(\tau(X_{2:n})\vert X_1) 
= \upprev{}_{\mathcal{M},\settrans{}}\big(\upprev{}_{\mathcal{M},\settrans{}}(\tau(X_{2:n})\vert X_{1:2}) \vert X_1\big) 
&= \upprev{}_{\mathcal{M},\settrans{}}\big(\overline{\Upsilon}(X_2)\vert X_1\big) \\
 &= [\uptrans{} \, \overline{\Upsilon}](X_1), 
\end{align*}
where the equalities follow from Theorem~\ref{Theorem: law of iterated}, Proposition~\ref{prop:time shift} and Equation~\eqref{Eq: Markov property}, respectively.
% Then $\upprev{}_{\mathcal{M},\settrans{}}(f(X_{1:n})\vert X_1)$ on its turn can 
% Then, using the properties of homogeneous imprecise Markov chains, we can efficiently compute $\upprev{}_{\mathcal{M},\settrans{}}(f(X_{1:n})\vert X_1)$ given the values of $\overline{\Upsilon}(y)\coloneqq\upprev{}_{\mathcal{M},\settrans{}}(\tau(X_{1:n-1})\vert X_1=y)$ for all $y \in \statespace{}$. 
% Indeed, fix any $y \in \statespace{}$ and note that $\prev_\mathrm{P}(f(X_{1:n})\vert X_1 = y) = g(y)+h(y) \prev_\mathrm{P}(\tau(X_{2:n})\vert X_1 = y)$ for all $\mathrm{P} \in \setofprocesses{{\mathcal{M},\settrans{}}}$, implying that $\upprev{}_{\mathcal{M},\settrans{}}(f(X_{1:n})\vert X_1 = y) = g(y)+h(y) \upprev{}_{\mathcal{M},\settrans{}}(\tau(X_{2:n})\vert X_1 = y)$, where, for the sake of simplicity, we assumed that $h$ does not take negative values.
% Subsequently, by appropriately combining Proposition \ref{prop:time shift} with Theorem \ref{Theorem: law of iterated}, one should then easily be able to come up with the following equality:
% \begin{align*}
% \upprev{}_{\mathcal{M},\settrans{}}(f(X_{1:n})\vert X_1 = y) = g(y)+h(y) [\uptrans{} \, \overline{\Upsilon}](y).
% \end{align*} 
So $\upprev{}_{\mathcal{M},\settrans{}}(f(X_{1:n})\vert X_1)$ can be obtained from $\overline{\Upsilon}$ by solving a single optimisation problem, followed by a pointwise multiplication and summation.
% Hence, the structure \eqref{Eq: structure function} of $f$ can be exploited to efficiently compute its upper expectation starting from the upper expectation of the function $\tau$, which, only depends on $n-1$ states instead of $n$.

Now, by repeating the structural assessment \eqref{Eq: structure function} in a recursive way, we can generate a whole class of functions for which the upper expectations can be computed using the principle illustrated above.
We start with a function $\tau_1(X_1)$, with $\tau_1\in\setofgambles{}(\statespace{})$, that only depends on the initial state. 
The upper expectation $\upprev{}_{\mathcal{M},\settrans{}}(\tau_1(X_1) \vert X_1)$ is then trivially equal to $\tau_1(X_1)$.
Next, consider $\tau_2(X_{1:2})=g_1(X_1)+h_1(X_1)\tau_1(X_{2})$ for some $g_1, h_1$ in $\setofgambles{}(\statespace{})$. 
$\upprev{}_{\mathcal{M},\settrans{}}(\tau_2(X_{1:2}) \vert X_1)$ is then given by $g_1(X_1)+h_1(X_1) [\uptrans{}\,\overline{\Upsilon}_1](X_{1})$, where we let $\overline{\Upsilon}_1(x)\coloneqq\upprev{}_{\mathcal{M},\settrans{}}(\tau_1(X_1) \vert X_1=x)=\tau_1(x)$ for all $x\in\statespace$ and where we (again) neglect the subtlety that $h_1$ can take negative values.
Continuing in this way, step by step considering new functions constructed by multiplication and summation with functions that depend on an additional time instance, and no longer ignoring the fact that the functions involved can take negative values, we end up with the following result.

\begin{theorem}\label{theorem: algorithm}
Consider any imprecise Markov chain $\setofprocesses{\mathcal{M},\settrans{}}$ 
% with upper and lower transition operator $\uptrans{}$ and $\lowtrans{}$ 
and two sequences of functions $\{g_n\}_{n \in \natz{}}$ and $\{h_n\}_{n \in \nats{}}$ in $\setofgambles{}(\statespace{})$. Define $\tau_1(x_1) \coloneqq g_0(x_1)$ for all $x_1 \in \statespace$, and for all $n\in\nats$, let
\begin{align*}
\tau_{n+1}(x_{1:n+1}) \coloneqq h_n(x_1) \tau_{n}(x_{2:n+1})+ g_n(x_1) &\text{ for all } x_{1:n+1} \in \statespace{}^{n+1}.
\end{align*}
If we write $\{\overline{\Upsilon}_n\}_{n \in \nats{}}$ and $\{\underline{\Upsilon}_n\}_{n \in \nats{}}$ to denote the sequences of functions in $\setofgambles{}(\statespace{})$ that are respectively defined by $\overline{\Upsilon}_n(x) \coloneqq \upprev{}_{\mathcal{M},\settrans{}}(\tau_n(X_{1:n}) \vert X_1 = x)$ and $\underline{\Upsilon}_n(x) \coloneqq \lowprev{}_{\mathcal{M},\settrans{}}(\tau_n(X_{1:n}) \vert X_1 = x)$ for all $x \in \statespace{}$ and all $n \in \nats{}$, then $\{\overline{\Upsilon}_n\}_{n \in \nats{}}$ and $\{\underline{\Upsilon}_n\}_{n \in \nats{}}$ satisfy the following recursive expressions:
\begin{align*}
\begin{cases}
\overline{\Upsilon}_1 = \underline{\Upsilon}_1 = g_0;\\
\overline{\Upsilon}_{n+1} = h_n \indica{h_n \geq 0}[\uptrans{} \, \overline{\Upsilon}_{n}] + h_n \indica{h_n < 0}[\lowtrans{} \, \underline{\Upsilon}_{n}] + g_n \text{ for all } n \in \nats{};\\
\underline{\Upsilon}_{n+1} = h_n \indica{h_n \geq 0}[\lowtrans{} \, \underline{\Upsilon}_{n}] + h_n \indica{h_n < 0}[\uptrans{} \, \overline{\Upsilon}_{n}] + g_n \text{ for all } n \in \nats{}.
\end{cases}
\end{align*}
\end{theorem}
Here, we used $\indica{h_n \geq 0} \in \setofgambles{}(\statespace{})$ to denote the indicator of $\{x \in \statespace{} \colon h_n(x) \geq 0\}$, and similarly for $\indica{h_n < 0} \in \setofgambles{}(\statespace{})$.
Note that, because we now need to evaluate both $\uptrans$ and $\lowtrans$ for every iteration, we will in general need to solve $2(n-1)\vert\statespace\vert$ optimisation problems to obtain $\upprev{}_{\mathcal{M},\settrans{}}(\tau_n(X_{1:n}) \vert X_1)$ and $\lowprev{}_{\mathcal{M},\settrans{}}(\tau_n(X_{1:n})\vert X_1)$ for some $n \in \nats{}$.
In order to obtain the unconditional inferences $\upprev{}_{\mathcal{M},\settrans{}}(\tau_n(X_{1:n}))$ and $\lowprev{}_{\mathcal{M},\settrans{}}(\tau_n(X_{1:n}))$, it then suffices to respectively maximise and minimise the expectations of $\upprev{}_{\mathcal{M},\settrans{}}(\tau_n(X_{1:n}) \vert X_1)$ and $\lowprev{}_{\mathcal{M},\settrans{}}(\tau_n(X_{1:n})\vert X_1)$ over all initial models in $\mathcal{M}$.
% If we are only interested in either the upper or the lower expectation, this is worse compared to the case of hitting times.
% If however, we are interested in both the upper and lower expectation, then we need $2n$ iterations in both cases.

\section{Special cases}

To illustrate the practical relevance of our method, we now discuss a number of important inferences that fall within its scope.
As already mentioned in the introduction section, in some of these cases, our method simplifies to a computational scheme that was already developed earlier in a more specific context.
The strength of our present contribution, therefore, lays in its unifying character and the level of generality to which it extends.

\vspace*{3pt}
% \paragraph*{}
\noindent
\textbf{Functions that depend on a single time instant.}
As a first, very simple inference we can consider the upper and lower expectation of a function $f(X_n)$, for some $f \in \setofgambles{}(\statespace{})$ and $n \in \nats{}$, conditional on the initial state. The expressions for these inferences are given by $\uptrans{}^{n-1} f$ and $\lowtrans{}^{n-1} f$, respectively~\cite{deCooman:2009jz}. For instance, for any $x\in\statespace$, $\lowprev{}_{\mathcal{M},\settrans{}}(f(X_5)) \vert X_1 = x)=[\lowtrans{}^{4}f](x)$. 
These expressions can also easily be obtained from Theorem~\ref{theorem: algorithm}, by setting $g_0 \coloneqq f$ and, for all $k \in \{1,\cdots,n-1\}$, $g_k \coloneqq 0$ and $h_k \coloneqq 1$.

\vspace*{3pt}
% \paragraph*{}
\noindent
\textbf{Sums of functions.}
One can also use our method to compute upper and lower expectations of sums $\sum_{k=1}^n f_k(X_k)$ of functions $f_k \in \setofgambles{}(\statespace{})$.
Then we would have to set $g_0 \coloneqq f_n$ and, for all $k \in \{1,\cdots,n-1\}$, $g_{k} \coloneqq f_{n-k}$ and $h_{k} \coloneqq 1$.
Although we allow the functions $f_k$ to depend on $k$, it is worth noting that, if we set them all equal to the same function $f$, our method can also be employed to compute the upper and lower expectation of the time average $\sfrac{1}{n} \sum_{k=1}^n f(X_k)$ of $f$ over the time interval $n$.
The subtlety of the constant factor $\sfrac{1}{n}$ does not raise a problem here, because upper and lower expectations are homogeneous with respect to non-negative scaling.% (since they are envelopes over linear expectations).

\vspace*{3pt}
% \paragraph*{}
\noindent
\textbf{Product of functions.}
Another interesting class of inferences are those that can be represented by a product $\prod_{k=1}^n f_k(X_k)$ of functions $f_k \in \setofgambles{}(\statespace{})$.
To compute upper and lower expectations of such functions, it suffices to set $g_0 \coloneqq f_n$ and, for all $k \in \{1,\cdots,n-1\}$, $g_{k} \coloneqq 0$ and $h_{k} \coloneqq f_{n-k}$.
A typical example of an inference than can be described in this way is the probability that the state will be in a set $A \subseteq \statespace{}$ during a certain time interval.
For instance, the upper expectation of the function $\indica{A}(X_1)\indica{A}(X_2)$ gives us a tight upper bound on the probability that the state will be in $A$ during the first two time instances.

\vspace*{3pt}
% \paragraph*{}
\noindent
\textbf{Hitting probabilities.}
The hitting probability of some set $A \subseteq \statespace{}$ over a finite time interval $n$ is the probability that the state $X_k$ will be in $A$ somewhere within the first $n$ time instances. 
The upper and lower bounds on such a hitting probability are equal to the upper and lower expectation of the function $f(X_{1:n}) \coloneqq \indica{A'_n} \in \setofgambles(\samplespace{})$, where $A'_n \coloneqq \{\omega \in \samplespace{} \colon (\exists k \leq n) \, \omega_k \in A \}$.
Note that $f(X_{1:n})$ can be decomposed in the following way:
\vspace*{-9pt}
\begin{align*}
f(X_{1:n}) = \indica{A}(X_1) + \indica{A}(X_2) \indica{A^c}(X_1) + \cdots + \indica{A}(X_n) \prod_{k = 1}^{n-1} \indica{A^c}(X_k) \vspace*{-10pt}
\end{align*} 
Hence, these inferences can be obtained using Theorem~\ref{theorem: algorithm} if we let $g_0 \coloneqq \indica{A}$ and, for all $k \in \{1,\cdots,n-1\}$, $g_k \coloneqq \indica{A}$ and $f_k \coloneqq \indica{A^c}$.
Additionally, one could also be interested in the probability that the state $X_k$ will \emph{ever} be in $A$. 
Upper and lower bounds on this probability are given by the upper and lower expectation of the function $f \coloneqq \indica{A'} \in \setofgambles(\samplespace{})$ where $A' \coloneqq \{\omega \in \samplespace{} \colon (\exists k \in \nats{}) \, \omega_k \in A \}$.
Since the function $f$ is non-finitary, we are unable to apply our method in a direct way.
However, it is shown in \cite[Proposition 16]{Krak2019HittingTimesISIPTA} that, if the set $\settrans{}$ is convex and closed, the upper and lower bounds on the hitting probability over a finite time interval converge to the upper and lower bounds on the hitting probability over an infinite time interval, therefore allowing us to approximate these inferences by choosing $n$ sufficiently large.

\vspace*{3pt}
% \paragraph*{}
\noindent
\textbf{Hitting times.}
The \emph{hitting time} of some set $A \subseteq \statespace{}$ is defined as the time $\tau$ until the state is in $A$ for the first time; so $\tau(\omega) \coloneqq \inf\{k \in \natz{} \colon \omega_{k} \in A \}$ for all $\omega \in \samplespace{}$.
Once more, the function $\tau$ is non-finitary, necessitating an indirect approach to the computation of its upper and lower expectation.
This can be done in a similar way as we did for the case of hitting probabilities, now considering the finitary functions $\tau_n(X_{1:n})$, where $\tau_n \in \setofgambles{}(\statespace{}^n)$ is defined by $\tau_n(x_{1:n}) \coloneqq \inf\{k \in \nats{} \colon x_{k} \in A\}$ if $\{k \in \nats{} \colon x_{k} \in A\}$ is non-empty, and $\tau_n(x_{1:n}) \coloneqq n+1$ otherwise, for all $n \in \nats{}$ and all $x_{1:n} \in \statespace{}^{n}$.
These functions correspond to choosing $g_0 \coloneqq \indica{A^c}$ and, for all $k \in \{1,\cdots,n-1\}$, $g_k \coloneqq \indica{A^c}$ and $f_k \coloneqq \indica{A^c}$.
If the set $\settrans{}$ is convex and closed, the upper and lower expectations of these functions for large $n$ will then approximate those of the non-finitary hitting time \cite[Proposition 10]{Krak2019HittingTimesISIPTA}.

\section{Discussion}

The main contribution of this paper is a single, unified method to efficiently compute a wide variety of inferences for imprecise Markov chains; see Theorem~\ref{theorem: algorithm}.
The set of functions describing these inferences is however restricted to the finitary type, and therefore a general approach for inferences characterised by non-finitary functions is still lacking.
In some cases, however, as we already mentioned in our discussion of hitting probabilities and hitting times, this issue can be addressed by relying on a continuity argument.

Indeed, consider any function $f = \lim_{n \to +\infty} \tau_n(X_{1:n})$ that is the pointwise limit of a sequence $\{\tau_n(X_{1:n})\}_{n \in \nats{}}$ of finitary functions, defined recursively as in Theorem~\ref{theorem: algorithm}.
If $\upprev{}_{\mathcal{M},\settrans{}}$ is continuous with respect to $\{\tau_n(X_{1:n})\}_{n \in \nats{}}$, meaning that $\lim_{n \to +\infty} \upprev{}_{\mathcal{M},\settrans{}}(\tau_n(X_{1:n})) = \upprev{}_{\mathcal{M},\settrans{}}( f)$, the inference $\upprev{}_{\mathcal{M},\settrans{}}(f)$ can then be approximated by $\upprev{}_{\mathcal{M},\settrans{}}(\tau_n(X_{1:n}))$ for sufficiently large $n$.
Since we can recursively compute $\upprev{}_{\mathcal{M},\settrans{}}(\tau_n(X_{1:n}))$ for any $n \in \nats{}$ using the methods discussed at the end of Section~\ref{section: computing inferences}, this yields an efficient way of approximating $\upprev{}_{\mathcal{M},\settrans{}}(f)$.
A completely analogous argument can be used for the lower expectation $\lowprev{}_{\mathcal{M},\settrans{}}(f)$.
This begs the question whether the upper and lower expectations $\upprev{}_{\mathcal{M},\settrans{}}$ and $\lowprev{}_{\mathcal{M},\settrans{}}$ satisfy the appropriate continuity properties for this to work.

Unfortunately, results about the continuity properties of these operators are rather scarce
% \footnote{It is well-known that these operators are coherent and therefore satisfy continuity with respect to uniform convergence, see \cite{Walley:1991vk,Troffaes:2014tl}. 
% However, this type of convergence is rather strong and is in most cases, such as for the cases of hitting probabilities and hitting times, not satisfied.}
---especially compared to their precise counterparts---and depend on the formalism that is adopted.
In this paper, for didactical reasons, we have considered one formalism: we have introduced imprecise Markov chains as being \emph{sets} of ``precise'' models that are in a specific sense compatible with the given set $\settrans$. 
It is however important to realise that there is also an entirely different formalisation of imprecise Markov chains that is instead based on the game-theoretic probability framework that was popularised by Shafer and Vovk; we refer to \cite{Shafer:2005wx,Tjoens2019NaturalExtensionISIPTA} for details.  
It is well known that the inferences produced under these two different frameworks agree for finitary functions \cite{deCooman:2008km,8535240}, so the method described by Theorem~\ref{theorem: algorithm} is also applicable when working in a game-theoretic framework.
The continuity properties of the game-theoretic upper and lower expectations, however, are not necessarily the same as those of the measure-theoretic operators that we considered here. So far, the continuity properties of game-theoretic upper and lower expectations are better understood \cite{Shafer:2005wx,Tjoens2019NaturalExtensionISIPTA,Tjoens2019ContinuityARXIV}, making these operators more suitable if we plan to employ the continuity argument above.
% As a closing remark, we want to mention that the limit behaviour of expected time averages is still largely unknown, both in a measure-theoretic and a game-theoretic context.
% It is currently one of our main topics of research since we expect that it would allow us to improve the ergodic theorem \cite[Theorem~32]{DECOOMAN201618}).
% For some preliminary empirical results regarding this topic, we refer to \cite[Section 7.6]{8535240}.

\newpage

\section*{Acknowledgments}
The work in this paper was partially supported by H2020-MSCA-ITN-2016 UTOPIAE, grant agreement 722734.

\bibliographystyle{plain}
\bibliography{references}

\appendix
\section{Appendix A}\label{appendix}

% \begin{lemma}[Markov Property]\label{lemma: Markov Property}
% 	Consider an imprecise Markov chain that is characterised by $\settrans{}$. 
% 	Then, for any $k,\ell \in \natz$ such that $k \leq \ell$ and any function $f \in \setofgambles{}(\statespace{}^{\ell - k +1})$, we have that 
% 	\begin{align*}
% 	\upprev{}_{\mathcal{M},\settrans{}}\big(f(X_{k:\ell}) \big\vert X_{0:m} = x_{0:m}\big) = \upprev{}_{\mathcal{M},\settrans{}}\big(f(X_{k:\ell}) \big\vert X_{m} = x_{m}\big),
% 	\end{align*}
% 	for all $m \leq k$ and all $x_{0:m} \in \statespace{}^{m+1}$.
% \end{lemma}

% Moreover, these upper expectation satisfy a number of technical properties that we will need further on; these are well-known in the literature and we simply report them here for convenience:
\begin{proposition}
Consider any imprecise Markov chain $\setofprocesses{\mathcal{M},\settrans{}}$. 
Then, for all $m,n\in\nats$ with $m \leq n$, all $f,g\in\setofgambles(\statespace^{n})$, all $x_{1:m}\in\statespace^{m}$, and all $\mu,\lambda\in\reals$ with $\lambda\geq 0$, it holds that
\begin{enumerate}[leftmargin=*,itemsep=2pt,ref={\upshape C\arabic*},label={\upshape C\arabic*}.]
	%\item $\upprev{}_{\mathcal{M},\settrans{}}(f(X_{m:n})\,\vert\,X_{0:k}=x_{0:k}) \leq \max_{x_{m:n}\in\statespace^{n-m+1}}f(x_{m:n})$
	%\item $\upprev{}_{\mathcal{M},\settrans{}}(f(X_{m:n}) + g(X_{m:n})\,\vert\,X_{0:k}=x_{0:k}) \leq \upprev{}_{\mathcal{M},\settrans{}}(f(X_{m:n})\,\vert\,X_{0:k}=x_{0:k}) + \upprev{}_{\mathcal{M},\settrans{}}(g(X_{m:n})\,\vert\,X_{0:k}=x_{0:k})$
	\item \label{coherence: Homogeneity} $\upprev{}_{\mathcal{M},\settrans{}}(\lambda f(X_{1:n})\,\vert\,X_{1:m}=x_{1:m}) = \lambda \upprev{}_{\mathcal{M},\settrans{}}( f(X_{1:n})\,\vert\,X_{1:m}=x_{1:m})$;
	\item \label{coherence: const. additiv.} $\upprev{}_{\mathcal{M},\settrans{}}(f(X_{1:n}) + \mu\,\vert\,X_{1:m}=x_{1:m}) = \upprev{}_{\mathcal{M},\settrans{}}( f(X_{1:n})\,\vert\,X_{1:m}=x_{1:m}) + \mu$;
	\item \label{coherence: monotonicity} $f \leq g \Rightarrow \upprev{}_{\mathcal{M},\settrans{}}(f(X_{1:n})\,\vert\,X_{1:m}=x_{1:m}) \leq \upprev{}_{\mathcal{M},\settrans{}}(g(X_{1:n})\,\vert\,X_{1:m}=x_{1:m})$;
	\item \label{coherence: conditioning} $\upprev{}_{\mathcal{M},\settrans{}}(f(X_{1:n})\,\vert\,X_{1:m}=x_{1:m}) = \upprev{}_{\mathcal{M},\settrans{}}(f(x_{1:m}X_{m+1:n})\,\vert\,X_{1:m}=x_{1:m})$;
	\item \label{coherence: const is const}
	$\upprev{}_{\mathcal{M},\settrans{}}(\mu\,\vert\,X_{1:m}=x_{1:m}) = \mu$,
\end{enumerate}
where in~\ref{coherence: conditioning}, if $m=n$, we let $f(x_{1:m}X_{m+1:n})\coloneqq f(x_{1:m})$.
\end{proposition}
\begin{proof}
\ref{coherence: Homogeneity}. This follows immediately from the definition of $\upprev{}_{\mathcal{M},\settrans{}}$ and the fact that both the supremum operator and, for all $\mathrm{P}\in\setofprocesses{\mathcal{M},\settrans{}}$, the expectation operator $\prev_\mathrm{P}[\cdot\,\vert\,\cdot]$, are homogeneous with respect to non-negative scaling. That is, because $\lambda\geq 0$, it holds that
\begin{align*}
\upprev{}_{\mathcal{M},\settrans{}}(\lambda f(X_{1:n})\,\vert\,X_{1:m}=x_{1:m}) &= \sup_{\mathrm{P}\in\setofprocesses{\mathcal{M},\settrans{}}} \prev_\mathrm{P}(\lambda f(X_{1:n})\,\vert\,X_{1:m}=x_{1:m}) \\
 &= \lambda\sup_{\mathrm{P}\in\setofprocesses{\mathcal{M},\settrans{}}} \prev_\mathrm{P}( f(X_{1:n})\,\vert\,X_{1:m}=x_{1:m}) \\
 &= \lambda \upprev{}_{\mathcal{M},\settrans{}}( f(X_{1:n})\,\vert\,X_{1:m}=x_{1:m})\,.
\end{align*}

\ref{coherence: const. additiv.}. This follows from the definition of $\upprev{}_{\mathcal{M},\settrans{}}$ together with the constant additivity of both the supremum operator and the expectation operators $\prev_\mathrm{P}[\cdot\,\vert\,\cdot]$ for all $\mathrm{P}\in\setofprocesses{\mathcal{M},\settrans{}}$. To wit,
\begin{align*}
\upprev{}_{\mathcal{M},\settrans{}}( f(X_{1:n})+\mu\,\vert\,X_{1:m}=x_{1:m}) &= \sup_{\mathrm{P}\in\setofprocesses{\mathcal{M},\settrans{}}} \prev_\mathrm{P}( f(X_{1:n})+\mu\,\vert\,X_{1:m}=x_{1:m}) \\
&= \mu + \sup_{\mathrm{P}\in\setofprocesses{\mathcal{M},\settrans{}}} \prev_\mathrm{P}( f(X_{1:n})\,\vert\,X_{1:m}=x_{1:m}) \\
&= \mu+ \upprev{}_{\mathcal{M},\settrans{}}( f(X_{1:n})\,\vert\,X_{1:m}=x_{1:m})\,.
\end{align*}

\ref{coherence: monotonicity}. For any $\mathrm{P}\in\setofprocesses{\mathcal{M},\settrans{}}$ the expectation operator $\prev_\mathrm{P}[\cdot\,\vert\,\cdot]$ is monotone, i.e., if $f\leq g$ then
\begin{equation*}
\prev_\mathrm{P}(f(X_{1:n})\,\vert\,X_{1:m}=x_{1:m}) \leq \prev_\mathrm{P}(g(X_{1:n})\,\vert\,X_{1:m}=x_{1:m})\,.
\end{equation*} 
Since this is true for all $\mathrm{P}\in\setofprocesses{\mathcal{M},\settrans{}}$, it immediately follows from the definition of $\upprev{}_{\mathcal{M},\settrans{}}$ that then also
\begin{align*}
\upprev{}_{\mathcal{M},\settrans{}}(f(X_{1:n})\,\vert\,X_{1:m}=x_{1:m}) &= \sup_{\mathrm{P}\in\setofprocesses{\mathcal{M},\settrans{}} }\prev_\mathrm{P}(f(X_{1:n})\,\vert\,X_{1:m}=x_{1:m}) \\ 
&\leq \sup_{\mathrm{P}\in\setofprocesses{\mathcal{M},\settrans{}} }\prev_\mathrm{P}(g(X_{1:n})\,\vert\,X_{1:m}=x_{1:m}) \\
 &= \upprev{}_{\mathcal{M},\settrans{}}(g(X_{1:n})\,\vert\,X_{1:m}=x_{1:m})\,.
\end{align*}
%Once more, the monotonicity of $\upprev{}_{\mathcal{M},\settrans{}}$ can easily be deduced from the definition of $\upprev{}_{\mathcal{M},\settrans{}}$ and the monotonicity of the linear expectation operator.

\ref{coherence: conditioning}. For any $\mathrm{P}\in\setofprocesses{\mathcal{M},\settrans{}}$, we know from the laws of probability that
\begin{align*}
\prev{}_{\mathrm{P}}(f(X_{1:n})\,\vert\,X_{1:m}=x_{1:m}) = \prev{}_{\mathrm{P}}(f(x_{1:m}X_{m+1:n})\,\vert\,X_{1:m}=x_{1:m}).
\end{align*}
The statement now follows directly from the definition of $\upprev{}_{\mathcal{M},\settrans{}}$.

\ref{coherence: const is const}. As before, this follows immediately from the definition of $\upprev{}_{\mathcal{M},\settrans{}}$.
\qed
\end{proof}

\noindent
Theorem~\ref{Theorem: law of iterated} is an imprecise counterpart of the well-known \emph{law of iterated expectations} for linear expectation operators, and was already presented in \cite[Theorem 7]{deCooman:2008km}.
However, the result is derived there in an entirely different framework, based on Walley's \cite{Walley:1991vk} notion of coherent lower and upper previsions.
For readers that are not familiar with that framework, it may therefore not be entirely clear how \cite[Theorem 7]{deCooman:2008km} implies our Theorem~\ref{Theorem: law of iterated}.
For that reason, we will instead start from \cite[Theorem 21]{8535240}, as it is expressed in the same measure-theoretic framework that we consider here.

Concretely, the authors of \cite[Theorem 21]{8535240} consider a setting in which, for all $n\in \nats{}$ and all $x_{1:n}\in\statespace{}^{n}$, a set of probability mass functions on $\statespace{}$, which we here denote by $\mathbb{P}_{x_{1:n}}$, is given. Furthermore, they also consider an additional set of probability mass functions on $\statespace{}$, which we here denote by $\mathbb{P}_\square$, that is used as the initial model.
They then let $\setofprocesses{\, \mathbb{P}}$ be the set of all possible precise global models $\mathrm{P}$ that can be obtained by choosing a probability mass function $\mathrm{P}(X_1)$ in $\mathbb{P}_\square$ and, for all $n\in \nats{}$ and all $x_{1:n}\in\statespace{}^{n}$, a probability mass function $\mathrm{P}( X_{n+1} \, \vert \, X_{1:n} = x_{1:n})$ in $\mathbb{P}_{x_{1:n}}$.
Then, according to \cite[Theorem 21]{8535240}, we have, for any $n,m \in \natz{}$ such that $n > m$ and any finitary function $g(X_{1:n})$, that
\begin{align}\label{Eq: Stavros iterated}
\upprev{}_{\mathbb{P}}\bigl(g(X_{1:n}) \big\vert X_{1:m} \bigr) 
= \lupprev{}\big(\lupprev{}\big( \cdots \lupprev{}\bigl(g(X_{1:n}) \big\vert X_{1:n-1} \bigr) \cdots \big\vert X_{1:m+1} \bigr) \big\vert X_{1:m} \big),
\end{align} 
where $\upprev{}_{\mathbb{P}}$ is obtained by maximising the linear expectation $\prev{}_{\mathrm{P}}$ over all precise global models $\mathrm{P} \in \setofprocesses{\, \mathbb{P}}$---or taking the supremum if no maximum is attained.
For our present purposes, an explicit definition of the operators $\lupprev{}(\, \cdot \, \vert X_{1:n})$ is not required; it suffices to know that they are completely determined by the sets of probability mass functions $\mathbb{P}_{x_{1:n}}$ and $\mathbb{P}_\square$

% It can easily be checked that Equation~\eqref{Eq: Stavros iterated} is equivalent with \cite[Theorem 21]{8535240}, which is formulated in terms of local lower expectations $\underline{{Q}}$, because of conjugacy and the definition of $\underline{{Q}}$\footnote{This is confirmed by the proof of \cite[Theorem 21]{8535240}.}.

In order to apply Equation~\eqref{Eq: Stavros iterated} in our setting, we now proceed to show that any imprecise Markov chain $\setofprocesses{\mathcal{M},\settrans{}}$, as defined in Section~\ref{section: imprecise markov chains}, is a special case of a set $\setofprocesses{\, \mathbb{P}}$ of global models as described above.
To that end, recall that in our setting, the imprecise Markov chain $\setofprocesses{\mathcal{M},\settrans{}}$ is the set of all global models $\mathrm{P}$ such that $\mathrm{P}(X_1)\in \mathcal{M}$ and, for all $n\in\nats$ and all $x_{1:n}\in\statespace^{n}$, there is some $T\in\settrans{}$ such that
\begin{equation*}
\probability{X_{n+1}=x_{n+1}}{X_{1:n}=x_{1:n}} = T(x_n,x_{n+1}) \text{ for all } x_{n+1} \in \statespace{},
\end{equation*}
or equivalently---since $\settrans{}$ has seperately specified rows---such that the local probability mass function $\probability{X_{n+1}}{X_{1:n}=x_{1:n}}$ is included in the set $\settrans{}_{x_n}$.% = \{T(x_n,\, \cdot \,) \colon T \in \settrans{}\}$.
It should therefore be clear that our model is indeed a special case of the one described above. It is obtained by setting $\mathbb{P}_{x_{1:n}}\coloneqq\settrans{}_{x_n}$ for all $n\in\nats$ and all $x_{1:n}\in\statespace^{n}$, and choosing $\mathbb{P}_\square \coloneqq \mathcal{M}$.
Consequently, for any $n,m \in \natz{}$ such that $n > m$ and any finitary function $g(X_{1:n})$, Equation~\eqref{Eq: Stavros iterated} immediately implies that
\begin{multline}\label{Eq: Stavros iterated 2}
\upprev{}_{\mathcal{M},\settrans{}}\bigl(g(X_{1:n}) \big\vert X_{1:m} \bigr) 
= \lupprev{}\big(\lupprev{}\big( \cdots \lupprev{}\bigl(g(X_{1:n}) \big\vert X_{1:n-1} \bigr) \cdots \big\vert X_{1:m+1} \bigr) \big\vert X_{1:m} \big).
\end{multline}
%where the operators $\lupprev{}(\cdot \vert X_{1:n})$ are now the one-step local upper expectations corresponding with $\mathcal{M}$ and $\settrans{}$.
Additionally, from the discussion above, it should be clear that \cite[Theorem 39]{8535240} and \cite[Lemma 40]{8535240} continue to hold in our setting.
By combining both results in a trivial way, we can conclude that Equation~\eqref{Eq: Markov property} indeed holds, as claimed in Section \ref{section: imprecise markov chains}.

\vspace{8pt}
\begin{proofof}{Theorem \ref{Theorem: law of iterated}}
Consider any $m \in \natz{}$, any $k \in \nats{}$ and any $f \in \setoffingambles{}(\samplespace{})$.
Then, without loss of generality, we can assume that $f$ depends on $X_{1:n}$, for some $n \in \nats{}$ such that $n > m+k$.
This allows us to write $g(X_{1:n}) \coloneqq f$ for some $g \in \setofgambles{}(\statespace{}^n)$.
So, it follows from Equation~\eqref{Eq: Stavros iterated 2} that
\begin{align*}
\upprev{}_{\mathcal{M},\settrans{}}\bigl(f \big\vert X_{1:m} \bigr)
&= \upprev{}_{\mathcal{M},\settrans{}}\bigl(g(X_{1:n}) \big\vert X_{1:m} \bigr)\\
&= \lupprev{}\big(\lupprev{}\big( \cdots \lupprev{}\bigl(g(X_{1:n}) \big\vert X_{1:n-1} \bigr) \cdots \big\vert X_{1:m+1} \bigr) \big\vert X_{1:m} \big)
\end{align*}
and, similarly, that
\begin{align*}
\upprev{}_{\mathcal{M},\settrans{}}\bigl(f \big\vert X_{1:m+k} \bigr)
&= \lupprev{}\big(\lupprev{}\big( \cdots \lupprev{}\bigl(g(X_{1:n}) \big\vert X_{1:n-1} \bigr) \cdots \big\vert X_{1:m+k+1} \bigr) \big\vert X_{1:m+k} \big).
\end{align*}
By combining both equalities, we find that
\begin{multline}\label{eq:innerreplaced}
\upprev{}_{\mathcal{M},\settrans{}}\bigl(f \big\vert X_{1:m} \bigr)\\
= \lupprev{}\big(\lupprev{}\big( \cdots \lupprev{}\bigl( \upprev{}_{\mathcal{M},\settrans{}}\bigl(f \big\vert X_{1:m+k} \bigr)\big\vert X_{1:m+k-1} \bigr)\cdots\big\vert X_{1:m+1}\bigr) \big\vert X_{1:m} \big).
\end{multline}

Now note that $\upprev{}_{\mathcal{M},\settrans{}}\bigl(f \big\vert X_{1:m+k} \bigr)$ is again a finitary function. Indeed, it clearly only depends on the state at a finite number of time instances---the first $m+k$, in this case---and is furthermore real-valued.
To understand the latter, observe that $f$ is bounded because $f$ can only take a finite number of values in $\reals{}$ (since $f$ is finitary and $\statespace$ is finite). 
Hence, $\inf f$ and $\sup f$ are real.
Then, by \ref{coherence: monotonicity} and \ref{coherence: const is const}, we can deduce that
\begin{align*}
\inf f \leq \upprev{}_{\mathcal{M},\settrans{}}\bigl(f \big\vert X_{1:m+k} \bigr) \leq \sup f,
\end{align*}
implying the real-valuedness of $\upprev{}_{\mathcal{M},\settrans{}}\bigl(f \big\vert X_{1:m+k} \bigr)$.

The final step of the proof consists in applying Equation~\eqref{Eq: Stavros iterated 2} one more time, but now to the finitary function $\upprev{}_{\mathcal{M},\settrans{}}\bigl(f \big\vert X_{1:m+k} \bigr)$. By doing so, we see that Equation~\eqref{eq:innerreplaced} is equivalent to
\begin{equation*}
 \upprev{}_{\mathcal{M},\settrans{}}\bigl(f \big\vert X_{1:m} \bigr)
= \upprev{}_{\mathcal{M},\settrans{}}\big( \upprev{}_{\mathcal{M},\settrans{}}\bigl(f \big\vert X_{1:m+k} \bigr) \big\vert X_{1:m} \big),
\end{equation*}
as desired.
\end{proofof}

\vspace{8pt}
\noindent
As a consequence of the law of iterated upper expectations, we can elegantly decompose the upper expectation of a finitary function using a notational trick that consists in extending the definition of the upper transition operator $\uptrans$. In particular, for any $n\in\nats{}$, we extend the operator $\uptrans\colon\setofgambles{}(\statespace{})\to\setofgambles{}(\statespace{})$ to an operator $\uptrans{} \colon \setofgambles{}(\statespace{}^{n+1}) \to \setofgambles{}(\statespace{}^{n})$, defined by
\begin{align*}
[\uptrans{}f](x_{1:n}) \coloneqq [\uptrans{}f(x_{1:n} \cdot)](x_{n}),
\end{align*}
for all $f \in \setofgambles{}(\statespace{}^{n+1})$ and all $x_{1:n} \in \statespace{}^n$, where we used the notation $f(x_{1:n} \cdot)$ to denote the function in $\setofgambles{}(\statespace{})$ that takes the value $f(x_{1:n+1})$ in $x_{n+1} \in \statespace{}$.
Using this notation, we obtain the following result.

\begin{proposition}\label{prop: iterated T}
Consider any imprecise Markov chain $\setofprocesses{\mathcal{M},\settrans{}}$.
Then for all $n,m\in\nats$ with $n>m$ and all $f\in\setofgambles{}(\statespace{}^{n-m+1})$, we have that
\begin{align}\label{Eq: prop: iterated T}
 \upprev{}_{\mathcal{M},\settrans{}}(f(X_{m:n}) \vert X_{1:m}) 
  &= \bigr[\, \overline{T}^{(n-m)} f \bigr] (X_{m}) .
\end{align}
\end{proposition}

 \begin{proof}
Fix any $m \in \nats{}$.
We prove by induction that Equation~\eqref{Eq: prop: iterated T} holds for all $n>m$.
It is clear that it holds for $n = m+1$ because of Equation~\eqref{Eq: Markov property}.
To prove the induction step, suppose that Equation~\eqref{Eq: prop: iterated T} holds for some $n>m$ and all $g\in\setofgambles{}(\statespace{}^{n-m+1})$.
Fix any $f\in\setofgambles{}(\statespace{}^{n-m+2})$.
Note that, for all $x_{1:n} \in \statespace{}^{n}$, we have that 
\begin{align*}
 \upprev{}_{\mathcal{M},\settrans{}}\bigl(f(X_{m:n+1}) \big\vert X_{1:n} = x_{1:n} \bigr) 
 &=  \upprev{}_{\mathcal{M},\settrans{}}\bigl(f(x_{m:n}X_{n+1}) \big\vert X_{1:n} = x_{1:n}\bigr) \nonumber \\
 &= \bigl[\uptrans{}(f(x_{m:n}\cdot)) \bigr](x_{n}) \nonumber \\
 &= \bigl[\uptrans{}f \bigr](x_{m:n}),
\end{align*}
where the first step follows from \ref{coherence: conditioning}, the second step from Equation~\eqref{Eq: Markov property} and the third from the extended definition of $\uptrans{}$.
Hence, we can write that 
\begin{align*}
% \label{Eq: proof prop: Vovk iterated T}
 \upprev{}_{\mathcal{M},\settrans{}}\bigl( \/ f(X_{m:n+1}) \big\vert X_{1:n} \bigr) 
 = \bigl[\uptrans{}f \bigr](X_{m:n}).
\end{align*}
Then, by Theorem~\ref{Theorem: law of iterated},
\begin{align*}
\upprev{}_{\mathcal{M},\settrans{}}\bigl(f(X_{m:n+1}) \big\vert X_{1:m} \bigr) 
&=\upprev{}_{\mathcal{M},\settrans{}}\bigl(\/ \upprev{}_{\mathcal{M},\settrans{}}\bigl(f(X_{m:n+1})\big\vert X_{1:n} \bigr) \big\vert X_{1:m} \bigr) \\
&=\upprev{}_{\mathcal{M},\settrans{}}\bigl([\uptrans{} f](X_{m:n}) \big\vert X_{1:m} \bigr).
\end{align*}
The last step consists in applying the induction hypothesis to find that
\begin{align*}
\upprev{}_{\mathcal{M},\settrans{}}\bigl(f(X_{m:n+1}) \big\vert X_{1:m} \bigr) 
&=\upprev{}_{\mathcal{M},\settrans{}}\bigl([\uptrans{} f](X_{m:n}) \big\vert X_{1:m} \bigr) \\
&= \Big[ \uptrans{}^{(n-m)} \bigl(\uptrans{} f \bigr) \Big] (X_m) = \bigl[ \uptrans{}^{(n-m+1)} f \bigr] (X_m),
\end{align*}
as desired. \qed
% If $m = n-1$, we are done.
% Otherwise, we have that $m < n-1$, and we can write, because of Theorem~\ref{Theorem: law of iterated}, that 
% \begin{align*}
% \upprev{}_{\mathcal{M},\settrans{}}\bigl(f(X_{m:n}) \big\vert X_{1:n-2} \bigr) 
% &=\upprev{}_{\mathcal{M},\settrans{}}\bigl(\upprev{}_{\mathcal{M},\settrans{}}\bigl(f(X_{m:n})\big\vert X_{1:n-1} \bigr) \big\vert X_{1:n-2} \bigr) \\
% &=\upprev{}_{\mathcal{M},\settrans{}}\bigl([\uptrans{} f](X_{m:n-1}) \big\vert X_{1:n-2} \bigr).
% \end{align*}
% Then, for all $x_{1:n-2} \in \statespace{}^{n-2}$, we similarly have that 
% \begin{align*}
%  \upprev{}_{\mathcal{M},\settrans{}}\Bigl([\uptrans{} f](X_{m:n-1}) &\Big\vert X_{1:n-2} = x_{1:n-2} \Bigr) \\  
%  &=   \upprev{}_{\mathcal{M},\settrans{}}\Bigl([\uptrans{} f](x_{m:n-2} X_{n-1}) \Big\vert X_{1:n-2} = x_{1:n-2} \Bigr) \\
%  &= \Big[ \uptrans{} \Big( [\uptrans{} f](x_{m:n-2} \cdot) \Big) \Big] (x_{n-2}) \\
%  &= \bigl[\uptrans{}^{\, (2)}f \bigr] (x_{m:n-2}) \nonumber
% \end{align*}
% If now $m = n-2$, this implies the desired equality.
% If not, we can continue to repeat the above argument to find that indeed $\upprev{}_{\mathcal{M},\settrans{}}(f(X_{m:n}) \vert X_{1:m}) 
%   = \bigr[\, \overline{T}^{(n-m)} f \bigr] (X_{m})$.
%   \qed
\end{proof}

\begin{proofof}{Proposition~\ref{prop:time shift}}
%This follows immediately from Proposition~\ref{prop: iterated T} which implies that both sides are equal to $[\overline{T}^{(n-m)}f](y)$.
By Proposition~\ref{prop: iterated T} it holds that
\begin{equation*}
\upprev{}_{\mathcal{M},\settrans{}}(f(X_{m:n}) \vert X_{1:m-1}=x_{1:m-1},X_m=y) 
= \bigr[\, \overline{T}^{(n-m)} f \bigr] (y). 
\end{equation*}
Similarly, it follows from Proposition~\ref{prop: iterated T} that
\begin{equation*}
\upprev{}_{\mathcal{M},\settrans{}}(f(X_{1:n-m+1}) \vert X_1=y) = \bigr[\, \overline{T}^{(n-m)} f \bigr] (y),
\end{equation*}
and hence
\begin{align*}
\upprev{}_{\mathcal{M},\settrans{}}(f(X_{m:n}) \vert X_{1:m-1}=x_{1:m-1},X_m=y) &= \bigr[\, \overline{T}^{(n-m)} f \bigr] (y) \\
 &= \upprev{}_{\mathcal{M},\settrans{}}(f(X_{1:n-m+1}) \vert X_1=y),
\end{align*}
as desired
\end{proofof}

\vspace{8pt}
\begin{proofof}{Theorem \ref{theorem: algorithm}}
Let us first recall that 
$\tau_1(x_1) = g_0(x_1)$ for all $x_1 \in \statespace$, and that, for any $n\in\nats$, 
\begin{align}\label{Eq: recursive expression tau}
\tau_{n+1}(x_{1:n+1}) = h_n(x_1) \tau_{n}(x_{2:n+1})+ g_n(x_1) &\text{ for all } x_{1:n+1} \in \statespace{}^{n+1}.
\end{align}
Furthermore, we have that $\overline{\Upsilon}_n(x) = \upprev{}_{\mathcal{M},\settrans{}}(\tau_n(X_{1:n}) \vert X_1 = x)$ and $\underline{\Upsilon}_n(x) = \lowprev{}_{\mathcal{M},\settrans{}}(\tau_n(X_{1:n}) \vert X_1 = x)$ for all $x \in \statespace{}$ and all $n \in \nats{}$.

To start our proof, we observe that, for any $x \in \statespace{}$, 
\begin{align*}
\overline{\Upsilon}_1(x) = \upprev{}_{\mathcal{M},\settrans{}}(g_0(X_{1}) \vert X_1 = x) 
= \upprev{}_{\mathcal{M},\settrans{}}(g_0(x) \vert X_1 = x) = g_0(x),
\end{align*} 
where we used \ref{coherence: conditioning} in the second step and \ref{coherence: const is const} in the last step. Similarly, by combining this argument with conjugacy, we find that
\begin{align*}
\underline{\Upsilon}_1(x) &= \lowprev{}_{\mathcal{M},\settrans{}}(g_0(X_{1}) \vert X_1 = x)\\
&=-\upprev{}_{\mathcal{M},\settrans{}}(-g_0(X_{1}) \vert X_1 = x)
=-\upprev{}_{\mathcal{M},\settrans{}}(-g_0(x) \vert X_1 = x)
= g_0(x).
\end{align*} 
Hence, we indeed have that $\overline{\Upsilon}_1 = \underline{\Upsilon}_1 = g_0$.

We now proceed to show that 
\begin{align}\label{eq:recursiveinproof}
\overline{\Upsilon}_{n+1} = h_n \indica{h_n \geq 0}[\uptrans{} \, \overline{\Upsilon}_{n}] + h_n \indica{h_n < 0}[\lowtrans{} \, \underline{\Upsilon}_{n}] + g_n \text{ for all } n \in \nats{}.
\end{align}
Fix any $n \in \nats{}$ and any $x \in \statespace{}$.
Using Equation~\eqref{Eq: recursive expression tau} in combination with \ref{coherence: conditioning}, we get that $\overline{\Upsilon}_{n+1}(x) = \upprev{}_{\mathcal{M},\settrans{}}(h_n(x) \tau_{n}(X_{2:n+1}) + g_n(x) \vert X_1 = x)$, which in turn is equal to $\upprev{}_{\mathcal{M},\settrans{}}(h_n(x) \tau_{n}(X_{2:n+1}) \vert X_1 = x) + g_n(x)$ because of \ref{coherence: const. additiv.}.
Furthermore, if $h_n(x) \geq 0$, then we have that $\overline{\Upsilon}_{n+1}(x) = h_n(x) \upprev{}_{\mathcal{M},\settrans{}}( \tau_{n}(X_{2:n+1}) \vert X_1 = x) + g_n(x)$ due to \ref{coherence: Homogeneity}.
If on the other hand $h_n(x) < 0$, then, by using \ref{coherence: Homogeneity} in combination with conjugacy, we find that
\begin{align*}
\overline{\Upsilon}_{n+1}(x) 
&= \upprev{}_{\mathcal{M},\settrans{}}( h_n(x)\tau_{n}(X_{2:n+1}) \vert X_1 = x) + g_n(x)\\
&= -h_n(x) \upprev{}_{\mathcal{M},\settrans{}}( -\tau_{n}(X_{2:n+1}) \vert X_1 = x) + g_n(x)\\
&= h_n(x) \lowprev{}_{\mathcal{M},\settrans{}}( \tau_{n}(X_{2:n+1}) \vert X_1 = x) + g_n(x).
\end{align*}
So in summary, we have found that
\begin{multline*}
\overline{\Upsilon}_{n+1}(x) = h_n(x) \indica{h_n \geq 0}(x) \upprev{}_{\mathcal{M},\settrans{}}( \tau_{n}(X_{2:n+1}) \vert X_1 = x) \\ 
+ h_n(x) \indica{h_n < 0}(x) \lowprev{}_{\mathcal{M},\settrans{}}( \tau_{n}(X_{2:n+1}) \vert X_1 = x) + g_n(x).
\end{multline*}
We now proceed to prove that 
$\upprev{}_{\mathcal{M},\settrans{}}( \tau_{n}(X_{2:n+1}) \vert X_1 = x)=[\uptrans{} \, \overline{\Upsilon}_{n}](x)$
and $\lowprev{}_{\mathcal{M},\settrans{}}( \tau_{n}(X_{2:n+1}) \vert X_1 = x)=[\lowtrans{} \, \underline{\Upsilon}_{n}](x)$. In combination with the above equality, this then clearly implies Equation~\eqref{eq:recursiveinproof}.

To see that 
$\upprev{}_{\mathcal{M},\settrans{}}( \tau_{n}(X_{2:n+1}) \vert X_1 = x)=[\uptrans{} \, \overline{\Upsilon}_{n}](x)$, we first apply Theorem~\ref{Theorem: law of iterated} to find that
\begin{align*}
\upprev{}_{\mathcal{M},\settrans{}}( \tau_{n}(X_{2:n+1}) \vert X_1 = x) &= \upprev{}_{\mathcal{M},\settrans{}}\big( \upprev{}_{\mathcal{M},\settrans{}}( \tau_{n}(X_{2:n+1}) \vert X_{1:2}) \vert X_1 = x \big).
\end{align*}
Now note that $\upprev{}_{\mathcal{M},\settrans{}}( \tau_{n}(X_{2:n+1}) \vert X_{1:2}) = \overline{\Upsilon}_{n}(X_2)$: indeed, for any $\omega \in \samplespace{}$, we have, by Proposition~\ref{prop:time shift}, that
\begin{align*}
\upprev{}_{\mathcal{M},\settrans{}}( \tau_{n}(X_{2:n+1}) \vert X_{1:2} = \omega_{1:2}) = \upprev{}_{\mathcal{M},\settrans{}}( \tau_{n}(X_{1:n}) \vert X_{1} = \omega_{2}) = \overline{\Upsilon}_{n}(\omega_2).
\end{align*}
Hence, we find that
\begin{align*}
\upprev{}_{\mathcal{M},\settrans{}}( \tau_{n}(X_{2:n+1}) \vert X_1 = x) &= \upprev{}_{\mathcal{M},\settrans{}}\big( \overline{\Upsilon}_{n}(X_2) \vert X_1 = x \big)=[\uptrans\,\overline{\Upsilon}_{n}](x),
\end{align*}
where we use Equation~\eqref{Eq: Markov property} for the final equality.

To see that also $\lowprev{}_{\mathcal{M},\settrans{}}( \tau_{n}(X_{2:n+1}) \vert X_1 = x)=[\lowtrans{} \, \underline{\Upsilon}_{n}](x)$, we employ a similar argument that additionally uses conjugacy.
We first apply Theorem~\ref{Theorem: law of iterated} to find that
\begin{align*}
\lowprev{}_{\mathcal{M},\settrans{}}( \tau_{n}(X_{2:n+1}) \vert X_1 = x)&=-\upprev{}_{\mathcal{M},\settrans{}}(-\tau_{n}(X_{2:n+1}) \vert X_1 = x) \\
&= -\upprev{}_{\mathcal{M},\settrans{}}\big( \upprev{}_{\mathcal{M},\settrans{}}( -\tau_{n}(X_{2:n+1}) \vert X_{1:2}) \vert X_1 = x \big).
\end{align*}
Now note that $\upprev{}_{\mathcal{M},\settrans{}}(-\tau_{n}(X_{2:n+1}) \vert X_{1:2}) = -\underline{\Upsilon}_{n}(X_2)$: indeed, for any $\omega \in \samplespace{}$, we have, by Proposition~\ref{prop:time shift}, that
\begin{align*}
\upprev{}_{\mathcal{M},\settrans{}}(-\tau_{n}(X_{2:n+1}) \vert X_{1:2} = \omega_{1:2}) &= \upprev{}_{\mathcal{M},\settrans{}}(- \tau_{n}(X_{1:n}) \vert X_{1} = \omega_{2})\\
&= -\lowprev{}_{\mathcal{M},\settrans{}}(\tau_{n}(X_{1:n}) \vert X_{1} = \omega_{2}) = -\underline{\Upsilon}_{n}(\omega_2).
\end{align*}
Hence, we find that
\begin{align*}
\lowprev{}_{\mathcal{M},\settrans{}}( \tau_{n}(X_{2:n+1}) \vert X_1 = x) = -\upprev{}_{\mathcal{M},\settrans{}}\big( -\underline{\Upsilon}_{n}(X_2) \vert X_1 = x \big)&=-[\uptrans(-\underline{\Upsilon}_{n})](x)\\
&=[\lowtrans\,\underline{\Upsilon}_{n}](x),
\end{align*}
where we use Equation~\eqref{Eq: Markov property} for the second equality.

% , the expression for $\overline{\Upsilon}_{n+1}(x)$ becomes 
% \begin{align*}
% \overline{\Upsilon}_{n+1}(x) 
% &= h_n(x) \indica{h_n \geq 0}(x) \upprev{}_{\mathcal{M},\settrans{}}( \overline{\Upsilon}_{n}(X_2) \vert X_1 = x) \\ 
% & \qquad \qquad + h_n(x) \indica{h_n < 0}(x) \lowprev{}_{\mathcal{M},\settrans{}}( \underline{\Upsilon}_{n}(X_2) \vert X_1 = x) + g_n(x) \\
% &= h_n(x) \indica{h_n \geq 0}(x) [\uptrans{} \, \overline{\Upsilon}_{n}](x) \\ 
% & \qquad \qquad \qquad+ h_n(x) \indica{h_n < 0}(x) [\lowtrans{} \, \underline{\Upsilon}_{n}](x) + g_n(x).
% \end{align*}
% Because this holds for any $n \in \nats{}$ and any $x \in \statespace{}$, this implies the claimed statement about $\overline{\Upsilon}_{n+1}$.

We are left to prove the recursive expression for $\underline{\Upsilon}_{n+1}$.
To that end, for all $n\in\nats{}$, we let $\tau'_n \coloneqq - \tau_n$ and we let $\smash{\overline{\Upsilon}_n'}\in \setofgambles{}(\statespace{})$ be defined by 
\begin{equation*}
\overline{\Upsilon}_n'(x)\coloneqq \upprev{}_{\mathcal{M},\settrans{}}(\tau'_n(X_{1:n}) \vert X_1 = x)
=-\lowprev{}_{\mathcal{M},\settrans{}}(\tau_n(X_{1:n}) \vert X_1 = x)=-\underline{\Upsilon}_n(x)
\end{equation*}
for all $x\in\statespace{}$. Similarly, for all $n\in\nats{}$, we let $\smash{\underline{\Upsilon}_n'}\in \setofgambles{}(\statespace{})$ be defined by
\begin{equation*}
\underline{\Upsilon}_n'(x)\coloneqq \lowprev{}_{\mathcal{M},\settrans{}}(\tau'_n(X_{1:n}) \vert X_1 = x)
=-\upprev{}_{\mathcal{M},\settrans{}}(\tau_n(X_{1:n}) \vert X_1 = x)=-\overline{\Upsilon}_n(x)
\end{equation*}
for all $x\in\statespace{}$.
Now observe that, for all $n \in \nats{}$,
\begin{align*}
\tau_{n+1}'(x_{1:n+1}) = h_n'(x_1) \tau_{n}'(x_{2:n+1}) + g_n'(x_1) &\text{ for all } x_{1:n+1} \in \statespace{}^{n+1},
\end{align*}
where we let $h_n' \coloneqq h_n$ and $g_n' \coloneqq - g_n$. 
It then follows from Equation~\eqref{eq:recursiveinproof} that
\begin{equation*}
\overline{\Upsilon}'_{n+1} = h_n' \indica{h_n' \geq 0} [\uptrans{} \, \overline{\Upsilon}'_{n}] 
+ h_n' \indica{h_n' < 0} [\lowtrans{} \, \underline{\Upsilon}'_{n}] + g_n'\text{ for all $n\in\nats{}$,}
\end{equation*}
or equivalently, that
\begin{align*}
-\underline{\Upsilon}_{n+1} &= h_n \indica{h_n \geq 0}(x) [\uptrans{}(-\underline{\Upsilon}_{n})] 
+ h_n \indica{h_n < 0} [\lowtrans{}(-\overline{\Upsilon}_{n})] - g_n\\
&= -h_n \indica{h_n \geq 0}(x) [\lowtrans{}\,\underline{\Upsilon}_{n}] 
- h_n \indica{h_n < 0} [\uptrans{}\,\overline{\Upsilon}_{n}] - g_n\text{ for all $n\in\nats{}$,}
\end{align*}
from which the desired recursive expression for $\underline{\Upsilon}_{n+1}$ follows immediately.
\end{proofof}
\end{document}